\documentclass[12pt]{amsart}

\usepackage[breaklinks,pdfstartview=FitH]{hyperref}
\usepackage{fullpage}
\usepackage{amssymb,enumerate}
\usepackage{pdfsync}


\theoremstyle{plain}
  \newtheorem{lemma}{Lemma}
  \newtheorem{theorem}[lemma]{Theorem}
  
  \newtheorem{claim}[lemma]{Claim}
  \newtheorem{proposition}[lemma]{Proposition}

 \theoremstyle{definition}

  \newtheorem{remark}[lemma]{Remark}

\newcommand{\e}{\varepsilon}
\DeclareMathOperator{\diam}{diam}
\newcommand{\vol}[1]{\left | #1 \right|}

\begin{document}

\title{On the Hausdorff dimension\\ of ultrametric subsets in~$\mathbb R^n$}

\author{James R. Lee}
\address{Computer Science \& Engineering, University of Washington} \email{jrl@cs.washington.edu}
\author{Manor Mendel}
\address{Mathematics and Computer Science, Open University of Israel} \email{mendelma@gmail.com}
\author{Mohammad Moharrami}
\address{Computer Science \& Engineering, University of Washington} \email{mohammad@cs.washington.edu}

\begin{abstract} For every $\e>0$, any subset of $\mathbb{R}^n$ with Hausdorff dimension larger than $(1-\e)n$ must have ultrametric distortion larger than $1/(4\e)$.
\end{abstract}

\maketitle

\noindent
We prove the following theorem.
\begin{theorem}
\label{thm:Rd-subsets}
For every $D>1$, every $n\in\mathbb N$, and every norm $\|\cdot\|$ on $\mathbb R^n$, any subset $S\subset \mathbb R^n$  having ultrametric distortion at most $D$,
must have Hausdorff dimension at most $\bigl(1-\frac{1}{2(D+1)}\bigr)n$.
\end{theorem}

An ultrametric space $(X,\rho)$  is a metric space satisfying $\rho(x,y)\le \max\{\rho(x,z),\rho(y,z)\}$ for all $x,y,z\in X$.
The {\em ultrametric distortion of a metric space $(X,d)$}, written $c_{\mathrm{UM}}(X,d)$, is the infimum over $D$ such that
there exists an ultrametric $\rho$ on $X$  satisfying $d(x,y)\le \rho(x,y)\le D\cdot d(x,y)$ for all $x,y\in X$.
The Euclidean distortion $c_2(X,d)$ of $(X,d)$ is defined similarly with respect to Hilbertian metrics over $X$.
The diameter of a metric space $(X,d)$ is given by $\diam(X)=\sup_{x,y\in X} d(x,y)$.
The $\alpha$-Hausdorff content of a metric space $(X,d)$ is defined as
\( \mathcal C^\alpha(X) = \inf \bigl\{ \sum_{i\in\mathbb N} \diam(A_i)^\alpha: \; \bigcup_{i\in\mathbb N} A_i \supseteq X \bigr\},\)
and the Hausdorff dimension of $X$ is $\dim_H(X)=\inf\{ \alpha>0:\ \mathcal{C}^\alpha(X)=0\}$.

Theorem~\ref{thm:Rd-subsets} proves that the Euclidean spaces $\mathbb R^n$ form (asymptotically)
tight examples
to the following Dvoretzky-type theorem for Hausdorff dimension from~\cite{MN11-ultra}.

\begin{theorem}[\cite{MN11-ultra}]
\label{thm:hausdorff-dvo}
For every $\e\in(0,1)$, every locally compact metric
space $(X,d)$ contains a subset $S\subseteq X$ having ultrametric distortion
at most $9/\e$, while having Hausdorff dimension at least $(1~-~\e)\dim_H(X)$.
\end{theorem}

Since separable ultrametrics embed isometrically in Hilbert space~\cite{TV},
Theorem~\ref{thm:hausdorff-dvo} is also true if one replaces ``ultrametric distortion"
with ``Euclidean distortion.''
Of course, for this (weaker) Euclidean version of Theorem~\ref{thm:hausdorff-dvo},
Euclidean spaces cannot serve as tight examples.
Tight examples for the Euclidean version of Theorem~\ref{thm:hausdorff-dvo}
are constructed in~\cite{MN11-ultra};
those spaces are stronger than $\mathbb R^n$ in the current context, but being ``fractals" based on expander graphs, they are also
more exotic.

\medskip

Previously,
Luosto~\cite{Luosto} proved a qualitative result along the lines of Theorem~\ref{thm:Rd-subsets}: Any subset  $S\subseteq \mathbb R^n$ of the $n$-dimensional Euclidean space which has finite ultrametric distortion  must have
$\dim_A(S)<n$, where $\dim_A(S)$ is the Assouad dimension of $S$ (note that
$\dim_A(X)\ge \dim_H(X)$ for every metric space $X$).
Luosto's proof gives only a weak quantitative bound on the Assouad dimension, namely,
$\dim_A(S)\le \bigl(1-\tfrac{c}{(2Dn)^n}\bigr) n$, for some universal constant $c>0$.
The proof of Theorem~\ref{thm:Rd-subsets} presented here is sufficiently flexible to derive a stronger version of
Theorem~\ref{thm:Rd-subsets}, with Assouad dimension replacing the Hausdorff dimension; see Remark \ref{rem:assouad}.
This variant of Theorem~\ref{thm:Rd-subsets} is an asymptotically tight quantitative version of Luosto's theorem.

It is not clear whether the constant $(1-1/(2(D+1))$ in
Theorem~\ref{thm:Rd-subsets}
is close to optimal when $D$ is large. 
However,  it is clear that Theorem~\ref{thm:Rd-subsets}
does not give  meaningful estimates when $D>1$ is small.
Luosto~\cite{Luosto} observed that the Hausdorff dimension of subsets $S\subset \mathbb R ^n$ must approach 0 as their ultrametric distortion approaches 1, i.e.,
for every $\delta >0$ there exists $\e>0$ such that if $c_{\text{UM}}(S)<1+\e$, then $\dim_H(S)<\delta$.
On the other hand, we have the following propostion.
\begin{proposition} \label{prop:low-dist}
For every $\e\in[0,1/4]$ and $n\in\mathbb{N}$, there exists $S\subset \mathbb R^n$ for which
$c_{\text{UM}}(S)\le 1+3\e$ and $\dim_H(S)\ge \frac{c\e^2}{\log(1/\e)}n$, for some universal $c>0$.
\end{proposition}
\begin{proof}[Sketch of a proof]
The argument is similar to~\cite[Lemma~8]{BBM}.
Take a binary code in $C\subset \{0,1\}^n$
of size $2^{c \e^2n}$  in which all pairwise Hamming distances are at the range
$\bigl[\frac{(1-\e)n}{2}, \frac{(1+\e)n}{2} \bigr]$. The set $S\subset \mathbb [0,1]^n$ is defined as
\(
S=\bigl \{ \sum_{i=0}^\infty (1-\e) \e^i x_i \; :\; x_i\in C\bigr\}. 
\)
\end{proof}

This property of $\mathbb R^n$ is qualitatively different from general metric spaces,
where there is an example of a compact metric space $X$
for which $\dim_H(X)=\infty$,
but for every subset $S\subset X$, if $c_{\text{UM}}(S)<2$, then $\dim_H(S)=0$; see~\cite{MN11-ultra,Fun11}.

\section*{Proof of Theorem~\ref{thm:Rd-subsets}}

Fix $D>1$, $n\in\mathbb N$ and a norm $\|\cdot \|$ on $\mathbb R^n$.
Denote by $B^o(r)=\{x\in \mathbb R^n: \|x\|<r\}$ the open ball of radius $r$ around the origin.
For subsets $A,B\subset \mathbb R^n$ we denote the Minkowski sum of $A$ and $B$ by $A+B=\{a+b:\, a\in A, \, b\in B\}$,
and for measurable sets $A$, we use $\vol{A}$ for the $n$-dimensional Lebesgue measure of $A$.

\begin{claim} \label{cl:approx-um}
Let $(X,d)$ a metric space that embeds in an ultrametric with distortion at most $D$, and let $x_0,\ldots, x_m\in X$.
Then
$\max_i d(x_i,x_{i-1})\ge d(x_0,x_m)/D$.
\end{claim}
\begin{proof}
Let $\rho$ be an ultrametric on $X$ such that ${d}\le \rho \le D\cdot d$.
We claim that $\max_i\rho(x_i,x_{i-1}) \ge \rho(x_0,x_m)$. Indeed, by induction
\begin{multline*}
\rho(x_0,x_m)\le \max\{\rho(x_0,x_1), \rho(x_1,x_m)\} \le \max\{\rho(x_0,x_1), \rho(x_1,x_2),\rho(x_2,x_m)\} \le \ldots
\\
\le \max\{\rho(x_0,x_1), \rho(x_1,x_2),\ldots, \rho(x_{m-1},x_m)\} .
\end{multline*}
Hence,
\begin{equation*}
{d(x_0, x_m)} \le \rho(x_0,x_m)
\le \max_i \rho(x_{i-1},x_i)\le D\cdot \max_i d(x_{i-1},x_i) .  \qedhere
\end{equation*}
\end{proof}

 \begin{claim} \label{cl:diam-bound}
 Let $S\subset \mathbb R^n$ be a subset that embeds in an ultrametric with distortion $D$.
 If $C$ is a path-connected subset of $S+B^o(r)$, then
 $\diam(C)\le 2(D+1)r$
 \end{claim}
 \begin{proof}
 Suppose for the sake of contradiction that $\diam(C)>2(D+1)r$.
 Fix $\eta>0$,
 and let $a_0,a_1\in C$ such that $\|a_0-a_1\|> 2(D+1)r$.
 Since $C$ is path-connected,
 there exists a continuous path $a:[0,1]\to C$ such that $a(0)=a_0$ and $a(1)=a_1$. Define $b:[0,1]\to S$, where $b(t)\in S$ is a point in $S$
 such that $\|b(t)-a(t)\|\le r$.  From the continuity of $a$ there exists a sequence of points $0=t_0<t_1<\ldots <t_m=1$ such that $\|a(t_i)-a(t_{i-1})\|\le \eta$ for every $i\in \{1,\ldots,m\}$.
 Hence $\|b(t_0)-b(t_m)\|>  2(D+1)r-2r$, and for every $i\in \{1,\ldots, m\}$,
 $\|b(t_i)-b(t_{i-1})\|\le  2r+2\eta$.
 But from Claim~\ref{cl:approx-um},
 \[ 2 r+2\eta \ge \max_i \|b(t_i)-b(t_{i-1})\| \ge \frac{\|b(t_0)-b(t_m)\|}{D}\,.
 \]
 Since the above is true for any $\eta>0$, we conclude that
 \[ 2r \ge  \frac{\|b(t_0)-b(t_m)\|}{D} > 2r\,,\]
contradicting our initial assumption.
\end{proof}

\begin{proof}[Proof of Theorem~\ref{thm:Rd-subsets}]
Suppose that $c_{\text{UM}}(S)\le D$.
We may assume without loss of generality that $S\subseteq B^o(1)$
(since one can find a countable subset $\mathcal N\subset \mathbb R^n$ such that
$\bigcup_{x\in \mathcal{N}} ((x+B^o(1))\cap S)=S$, and for any countable collect of subsets
$\{A_x\}_{x\in\mathcal N}$ we have $\dim_H(\cup_{x\in \mathcal N} A_x)=\sup_{x\in\mathcal N} \dim_H(A_x)$ ).
Fix $\delta>0$, fix $r>0$,
and fix a path-component  $C\subset S+B^o(e^\delta r)$ of $S+B^o(e^\delta r)$.
Note that $C$ is an open subset.
 By Claim~\ref{cl:diam-bound}, $\diam(C)\le 2(D+1)e^\delta r$, and hence
$\vol{C}\le \vol{B^o(2(D+1)e^\delta r)}$, which means that
$\vol{B^o((e^\delta-1)r)}\ge \left(\frac{e^\delta -1}{2(D+1)e^\delta} \right)^{n} \vol{C}$.
Let $A=(S\cap C)+B^o(r)$.
 Observe that $C=A+B^o((e^\delta-1)r)$, and that $A$, $B^o((e^\delta-1)r)$, and $C$ are bounded and open.  By the Brunn-Minkowski inequality,
\begin{equation} \label{eq:component-vol-ineq}
\vol{A}^{1/n} \le \vol{C}^{1/n}- \vol{B^o((e^\delta -1)r)}^{1/n} \le \left (1-\frac{e^\delta -1}{2(D+1)e^\delta} \right) \vol{C}^{1/n}\,.
\end{equation}
Since the path-components of $S+B^o(e^\delta r)$ are open, and those components constitute a pairwise disjoint cover of $S+B^o(e^\delta r)$, by summing the $n$-th power of~\eqref{eq:component-vol-ineq}
over the path-components of $S+B^o(e^\delta r)$, we obtain
\begin{equation} \label{eq:vol-ineq}
\vol{S+B^o(r)} \le \left(1-\frac{e^\delta -1}{2(D+1)e^\delta}\right)^n \vol{S+B^o(e^\delta r)}.
\end{equation}

Fix $\alpha>\left(1+\delta^{-1} \log\left(1-\frac{e^{\delta}-1}{2(D+1)e^\delta}\right)\right) n$.
We will prove that $\mathcal C^\alpha(S)=0$ by constructing a
sequence of covers of $S$.
The $j$-th cover of $S$ is the set of path-components of $S+B^o(e^{-\delta j})$.
Let $\beta=\vol{B^o(1)}>0$. Note that $S+B^o(1)\subset B^o(2)$, and hence $\vol{S+B^o(e^{\delta 0})}\le 2^n \beta$.
By inductively applying~\eqref{eq:vol-ineq},
$\vol{S+B^o(e^{-\delta j})} \le
\left(1-\frac{e^\delta -1}{2(D+1)e^\delta}\right)^{jn} 2^n\beta$. On the other hand, each path-component
 of $S+B^o(e^{-\delta j})$ has a volume at least $\vol{B^o(e^{-\delta j})}=e^{-\delta jn} \beta$. Therefore, $S+B^o(e^{-\delta j})$ has at most
 $2^n \left(e^\delta \left (1-\frac{e^\delta -1}{2(D+1)e^\delta}\right)\right )^{jn}$ path-components, and by Claim~\ref{cl:diam-bound} each of the components has diameter at most
 $2(D+1)e^{-\delta j}$. Hence,
 \begin{multline*}
  \mathcal C^{\alpha}(S)\le 2^n \left(e^\delta \left (1-\frac{e^\delta -1}{2(D+1)e^\delta}\right)\right )^{jn} \cdot (2(D+1)e^{-\delta j})^{\alpha}
\\
\le 2^n (4D)^\alpha \cdot \left ( \frac{e^\delta \left (1-\frac{e^\delta -1}{2(D+1)e^\delta}\right)}{e^{\delta\alpha/n}}\right)^{jn} \xrightarrow[j\rightarrow \infty]{}0\,,
 \end{multline*}
 and therefore $\dim_H(S)\le \alpha$. Since the preceding bound is true for every $\delta > 0$ and every
 $\alpha>\left(1+\delta^{-1} \log\left(1-\frac{e^{\delta}-1}{2(D+1)e^\delta}\right)\right) n$, we conclude that
\[ \dim_H(S) \le \lim_{\delta\to 0^+} \left(1+\delta^{-1} \log\left(1-\frac{e^{\delta}-1}
{2(D+1)e^\delta}\right)\right) n
 =\left(1-\tfrac{1}{2(D+1)}\right)n\,.  \qedhere\]
\end{proof}

\begin{remark}\label{rem:assouad}
One may obtain the same conclusion with Assouad dimension replacing Hausdorff dimension.
This follows from
the fact that we have a uniform bound on the diameter of the elements in
our cover at every step; hence the same sequence of covers shows that
$\dim_A(S) \leq \left(1-\tfrac{1}{2(D+1)}\right)n$ as well.  We leave
verification as an exercise for the interested reader.
\end{remark}

\subsection*{Acknowledgments}
The authors wish to thank Uzy Hadad, Assaf Naor, and Yuval Peres for helpful and constructive discussions about Theorem~\ref{thm:Rd-subsets}.
M.~Me.  was partially supported by ISF grants 221/07, 93/11,
and BSF grants 2006009, 2010021.
This work was carried out while he was visiting  Microsoft Research
and University of Washington.
J.~L. and M.~Mo. were partially supported
by NSF grant CCF-0915251 and a Sloan Research Fellowship.

\bibliographystyle{abbrvurl}
\bibliography{hausdorff}

\end{document}